\documentclass[11pt]{amsart}

\scrollmode

\usepackage{times}
\usepackage{amsmath,graphicx}
\usepackage{latexsym}
\usepackage{amscd,amsthm,amssymb}
\usepackage[all]{xy}

\newtheorem{thm}{Theorem}[section]
\newtheorem{prop}[thm]{Proposition}
\newtheorem{lem}[thm]{Lemma}
\newtheorem{cor}[thm]{Corollary}

\theoremstyle{definition}

\theoremstyle{remark}
\newtheorem{rem}[thm]{Remark}

\numberwithin{equation}{section}

\title{Iterated fibre sums of algebraic Lefschetz fibrations}
\author{M.~J.~D.~Hamilton}
\address{      Institute for geometry and topology\\
               University of Stuttgart\\
               Pfaffenwaldring 57\\
               70569 Stuttgart\\
               Germany}
\email{mark.hamilton@math.lmu.de}
\date{\today}
\subjclass[2010]{Primary 57R57, 57R17; Secondary 14J29}
\keywords{4-manifold, symplectic, fibre sum, Lefschetz fibration}

\begin{document}

\begin{abstract}Let $M$ denote the total space of a Lefschetz fibration, obtained by blowing up a Lefschetz pencil on an algebraic surface. We consider the $n$-fold fibre sum $M(n)$, generalizing the construction of the elliptic surfaces $E(n)$. For a Lefschetz pencil on a simply-connected minimal surface of general type we partially calculate the Seiberg-Witten invariants of the fibre sum $M(n)$ using a formula of Morgan-Szab\'o-Taubes. As an application we derive an obstruction for self-diffeomorphisms of the boundary of the tubular neighbourhood of a general fibre in $M(n)$ to extend over the complement of the neighbourhood. Similar obstructions are known in the case of elliptic surfaces.
\end{abstract}

\maketitle

\section{Introduction}

It is well-known that the sequence of simply-connected elliptic surfaces $E(n)$ without multiple fibres can be constructed from the elliptic surface $E(1)$, diffeomorphic to $\mathbb{CP}^2\#9{\overline{\mathbb{CP}}{}^{2}}$, by fibre summing along general fibres \cite[Section 3.1]{GS}. The elliptic fibration on $E(1)$ can be obtained by considering a certain Veronese embedding of $\mathbb{CP}^2$ into a complex projective space and then taking the blow-up of a Lefschetz pencil. We consider the following generalization: Let $M'$ denote an arbitrary smooth algebraic surface. It admits a Lefschetz pencil which extends to a Lefschetz fibration  on some blow-up $M$. We can consider the iterated fibre sums of these fibrations, yielding a sequence $M(n)$ of symplectic manifolds with an induced Lefschetz fibration over $\mathbb{CP}^1$. 

We describe the basic topology of these manifolds in Section \ref{sect invariants}. For example, if $M$ is simply-connected, then all $M(n)$ are simply-connected and there is a description of the intersection form and of the canonical class. Considering the Seiberg-Witten invariants of $M(n)$, we show that if $M'$ is a minimal surface of general type, then the only basic class up to sign of $M(n)$ for all $n\geq 2$ which has non-zero intersection with the fibre is the canonical class. To determine the Seiberg-Witten basic classes of the fibre sum $M(n)$ we use a formula of Morgan, Szab\'o and Taubes. In general, this formula does not determine the Seiberg-Witten invariant of a single characteristic class, but involves a sum over several classes. However, in the case of the manifolds $M(n)$ there is only one summand and the formula completely determines the basic classes under the mentioned constraint.   

We then give an application of these calculations to the question which orientation preserving self-diffeomorphisms of the boundary of the tubular neighbourhood $\nu\Sigma$ of a general fibre in $M(n)$ extend over the complement $M(n)\setminus \text{int}\,\nu\Sigma$. In the case of elliptic surfaces $E(n)$, a complete answer to the corresponding question is known \cite[Theorem 8.3.11]{GS}: For $E(1)$ every diffeomorphism on $\partial\nu\Sigma$ extends over $E(1)\setminus \text{int}\,\nu\Sigma$ and for $E(n)$ with $n\geq 2$ only if it preserves the torus fibration on the boundary. For the manifolds $M(n)$ we derive a criterion that can be used to show that certain diffeomorphisms do not extend over the complement.

\section{Lefschetz fibrations}\label{sect Lefschetz}
Let $M'$ be a smooth complex algebraic surface in some $\mathbb{CP}^N$ of degree $r$, so that it represents the class 
\begin{equation*}
[M']=r[\mathbb{CP}^2]\in H_4(\mathbb{CP}^N;\mathbb{Z}).
\end{equation*}
Choose a hyperplane $H\cong\mathbb{CP}^{N-1}$ in $\mathbb{CP}^N$ intersecting $M'$ transversely. According to Section 1.5 and 1.6 in \cite{La} there exists a Lefschetz pencil containing $H$ as one of its hyperplanes. A pencil is the set of hyperplanes which contain a given linear subspace $A\cong\mathbb{CP}^{N-2}$, called the axis of the pencil. The axis $A$ intersects $M'$ transversely in $r$ points, forming the base locus $B$. Blowing up these points results in a fibration
\begin{equation*}
M=M'\#r{\overline{\mathbb{CP}}{}^{2}}\longrightarrow\mathbb{CP}^1.
\end{equation*}
This is what we call an {\em algebraic Lefschetz fibration}. A Lefschetz fibration has finitely many singular fibres, where the singularities have a certain normal form. We can assume that each singular fibre has precisely one singularity. The generic fibre $\Sigma_M$ is given by the proper transform of a non-singular hyperplane section $\Sigma_{M'}$ in $M'$ determined by a generic hyperplane in the pencil. We denote the genus of $\Sigma_M$ by $g$. 

By the first Lefschetz Hyperplane theorem, the homomorphism
\begin{equation*}
i_M\colon H_1(\Sigma_M;\mathbb{Z})\rightarrow H_1(M;\mathbb{Z}),
\end{equation*}
induced by inclusion for a smooth fibre $\Sigma_M$ is a surjection. The so-called second Lefschetz Hyperplane theorem \cite{AF} shows that the kernel of this map is generated by the set of {\em vanishing cycles}. The vanishing cycles bound embedded disks in $M$, called {\em Lefschetz thimbles} or {\em vanishing disks}, which intersect $\Sigma_M$ only in the vanishing cycle and contain precisely one critical point of the fibration. The vanishing disks are formed above certain arcs in $\mathbb{CP}^1$. For each critical point there is a corresponding vanishing cycle and a vanishing disk. If we frame a vanishing disk $D$ on its boundary by the direction normal to the cycle inside the fibre, then $D$ has self-intersection equal to $-1$.

The fibration defines a natural framing for the tubular neighbourhood of a general fibre $\Sigma_M$ in $M$. Consider a diffeomorphism between the fibres in two copies of $M$ that identifies the vanishing cycles. Lift the diffeomorphism in the standard way to an orientation reversing diffeomorphism of the boundary of the tubular neighbourhoods, using the framing determined by the fibration. If we form the generalized fibre sum, we get a closed  symplectic 4-manifold
\begin{equation*}
M(2)=M\#_{\Sigma_M=\Sigma_M}M.
\end{equation*}
We can iterate the construction to get symplectic 4-manifolds $M(n)$, where
\begin{equation*}
M(n)=M\#_{\Sigma_M=\Sigma_M}M\#_{\Sigma_M=\Sigma_M}\ldots\#_{\Sigma_M=\Sigma_M}M.
\end{equation*}
By our choice of gluing, the Lefschetz fibration on $M$ extends to a symplectic Lefschetz fibration on $M(n)$.

\section{Fibre sums and the Morgan-Szab\'o-Taubes formula}\label{fibre sum}

We recall some results about generalized fibre sums \cite{MHthesis, MH}. Let $M$ and $N$ denote closed oriented 4-manifolds with closed oriented embedded surfaces $\Sigma_M$ and $\Sigma_N$ of genus $g$ and self-intersection zero. We choose embeddings
\begin{align*}
i_M&\colon\Sigma\rightarrow M\\
i_N&\colon\Sigma\rightarrow N
\end{align*}
that realize the surfaces as images of a closed oriented surface $\Sigma$. We fix framings $\Sigma\times D^2$ of the closed tubular neighbourhoods $\nu\Sigma_M$ and $\nu\Sigma_N$ and denote the manifolds minus the interior of the neighbourhoods by $M_0$ and $N_0$. We want to glue $M_0$ and $N_0$ together using an orientation reversing diffeomorphism $\phi\colon\partial M_0\rightarrow\partial N_0$ that preserves the circle fibration and covers the diffeomorphism $i_N\circ i_M^{-1}$ between the surfaces. In the chosen framings any such diffeomorphism is isotopic to a diffeomorphism of the form
\begin{align*}
\phi\colon\Sigma\times S^1&\longrightarrow\Sigma\times S^1,\\
(x,\alpha)&\mapsto (x,C(x)\cdot\overline{\alpha})
\end{align*}
where the bar denotes complex conjugation and $C\colon\Sigma\rightarrow S^1$ is a smooth map. The map $\phi$ depends up to isotopy only on the cohomology class $C^*d\alpha\in H^1(\Sigma;\mathbb{Z})$ that we also denote by $C$. The result of the generalized fibre sum is
\begin{equation*}
X=M\#_{\Sigma_M=\Sigma_N}N=M_0\cup_\phi N_0
\end{equation*}
and in general depends on the choice of the cohomology class $C$.

It is well-known that the Euler characteristic and the signature of $X$ are given by
\begin{align*}
e(X)&=e(M)+e(N)+4g-4\\
\sigma(X)&=\sigma(M)+\sigma(N).
\end{align*}
Suppose that $\Sigma_M$ and $\Sigma_N$ represent indivisible homology classes. Then the first homology of $X$ is given by the cokernel of the homomorphism
\begin{equation*}
i_M\oplus i_N\colon H_1(\Sigma;\mathbb{Z})\rightarrow H_1(M;\mathbb{Z})\oplus H_1(N;\mathbb{Z}),
\end{equation*}
induced by the embeddings.

We want to describe the second homology of $X$. Assume from now on that $M$, $N$ and $X$ have torsion free homology and $\Sigma_M$ and $\Sigma_N$ represent indivisible classes. The framings determine push-offs of $\Sigma_M$ and $\Sigma_N$ into the boundaries $\partial M_0$ and $\partial N_0$. Under inclusion in $X$ we get surfaces $\Sigma_X$ and $\Sigma_X'$. There also exist surfaces $B_M$ and $B_N$ in $M$ and $N$ which intersect $\Sigma_M$ and $\Sigma_N$ in a single positive transverse point. Since the gluing preserves the meridians $\{\ast\}\times S^1$ these surfaces sew together in $X$ to a surface $B_X$. We define $P(M)$ to be the orthogonal complement of the span of $\Sigma_M$ and $B_M$ in $H_2(M;\mathbb{Z})$, and similarly for $N$. A curve on $\Sigma_M$ times the meridian $\sigma_M$ defines a torus on $\partial M_0$. Under inclusion in $X$ this is a torus of self-intersection zero, called a {\em rim torus}. It is null-homologous in $M$, but not necessarily in $X$. There are also so called {\em vanishing surfaces} in $X$ (see \cite{FS2}), sewed together along a curve on the push-off of $\Sigma_M$ in $\partial M_0$ and a curve on $\partial N_0$ which get identified under gluing and bound in $M_0$ and $N_0$. With these preparations we can describe the second homology of $X$ and the intersection form. Let $c$ denote the rank of the kernel of the homomorphism $i_M\oplus i_N$ above.  

\begin{thm}\label{formula H^2 no torsion} There exists a basis $S_1,\ldots,S_c$ of the vanishing classes $S'(X)$ and a basis $R_1,\ldots R_c$ of the rim tori $R(X)$ such that there exists a splitting
\begin{equation*}
H_2(X;\mathbb{Z})=P(M)\oplus P(N)\oplus (S'(X)\oplus R(X))\oplus (\mathbb{Z}B_X\oplus\mathbb{Z}\Sigma_X),
\end{equation*}
where 
\begin{equation*}
(S'(X)\oplus R(X))=(\mathbb{Z}S_1\oplus\mathbb{Z}R_1)\oplus\dotsc\oplus(\mathbb{Z}S_c\oplus\mathbb{Z}R_c).
\end{equation*}
The direct sums are all orthogonal, except the direct sums inside the brackets. In this decomposition of $H_2(X;\mathbb{Z})$, the restriction of the intersection form $Q_X$ to $P(M)$ and $P(N)$ is equal to the intersection form induced from $M$ and $N$ and has the structure
\begin{equation*}
\left(\begin{array}{cc} B_M^2+B_N^2 &1 \\ 1& 0 \\ \end{array}\right)
\end{equation*} 
on $\mathbb{Z}B_X\oplus\mathbb{Z}\Sigma_X$ and the structure   
\begin{equation*}
\left(\begin{array}{cc} S_i^2 &1 \\ 1& 0 \\ \end{array}\right)
\end{equation*}  
on each summand $\mathbb{Z}S_i\oplus\mathbb{Z}R_i$. We call such a basis for $H_2(X;\mathbb{Z})$ a normal form basis.
\end{thm}
The vanishing surfaces $S_i$ are obtained from a basis $\alpha_1,\ldots,\alpha_c$ for the kernel of $i_M\oplus i_N$. Suppose we only have a generating set for this subgroup. Expressing each basis element $\alpha_i$ in terms of the generating set we have:
\begin{prop}\label{gen set for S'X}
Let $\tilde{S}_1,\ldots,\tilde{S}_e$ denote vanishing surfaces obtained from a generating set for the kernel of $i_M\oplus i_N$. Then we can write each vanishing surface $S_i$ in Theorem \ref{formula H^2 no torsion} as a linear combination of the $\tilde{S}_j$ and certain rim tori.
\end{prop}
The rim tori are needed in this proposition to separate the different vanishing surfaces that we get. Generally speaking, we fix a basis of rim tori and then determine a dual basis of vanishing surfaces. A basis of rim tori in one generalized fibre sum determines such a basis in all of them.

There is a formula for the Seiberg-Witten invariants of $X$ due to Morgan, Szab\'o and Taubes \cite{MST} that we want to describe in our notation. The formula works only if $g\geq 2$. Recall that the Seiberg-Witten invariant of a closed, oriented 4-manifold with $b_2^+>1$ is a map
\begin{equation*}
SW_X\colon \mathrm{Spin}^c(X)\rightarrow\mathbb{Z}.
\end{equation*}
There is a related invariant
\begin{equation*}
SW_X\colon\mathcal{C}(X)\rightarrow\mathbb{Z},
\end{equation*}
defined on the set $\mathcal{C}(X)$ of characteristic elements in $H^2(X;\mathbb{Z})$ by summing over all $\mathrm{Spin}^c$-structures with the same first Chern class. Under our assumptions a $\mathrm{Spin}^c$-structure is determined by its Chern class. A characteristic class $k\in \mathcal{C}(X)$ is called a Seiberg-Witten basic class if $SW_X(k)\neq 0$. Let $k'$ be a Seiberg-Witten basic class of $X$. The adjunction inequality shows that $k'$ has zero intersection with each rim torus, hence the Poincar\'e dual of $k'$ has no vanishing surface component. The adjunction inequality also shows that $|k'\cdot\Sigma_X|\leq 2g-2$. The formula only makes a statement about the case that $k'\cdot\Sigma_X=\pm(2g-2)$. We can restrict to the positive case: Let $k$ denote a characteristic class on $X$ such that the Poincar\'e dual $PD(k)$ is of the form
\begin{equation*}
PD(k)=p_M+p_N+\sum_{i=1}^c\epsilon_iR_i+(2g-2)B_X+\beta_X\Sigma_X,
\end{equation*}
with $p_M\in P(M)$ and $p_N\in P(N)$. One can show that
\begin{equation*}
H^2(M_0;\mathbb{Z})\cong H_2(M_0,\partial M_0;\mathbb{Z})=P(M)\oplus \mathbb{Z}B_M'\oplus\text{ker}\,i_M,
\end{equation*}
where $B_M'$ is the surface $B_M$ with a disk deleted and $i_M$ denotes the map on $H_1(\Sigma;\mathbb{Z})$ induced by inclusion. There are obvious restriction maps of $H^2(X;\mathbb{Z})$ and $H^2(M;\mathbb{Z})$ to $H^2(M_0;\mathbb{Z})$, and similarly for $N$. It follows that the set of characteristic classes on $X$, which have the same square and the same restriction to $M_0$ and $N_0$ as $k$, is given by
\begin{equation*}
\mathcal{K}(k)=\{l\in \mathcal{C}(X)\mid \text{$PD(l)=PD(k)+R$ with $R\in R(X)$}\}.
\end{equation*}
We also set
\begin{align*}
\mathcal{K}_M(k)&=\{l\in\mathcal{C}(M)\mid \text{$PD(l)=p_M+(2g-2)B_M+\beta_M\Sigma_M$ with $\beta_M\in\mathbb{Z}$}\}\\
\mathcal{K}_N(k)&=\{l\in\mathcal{C}(N)\mid \text{$PD(l)=p_N+(2g-2)B_N+\beta_N\Sigma_N$ with  $\beta_N\in\mathbb{Z}$}\}.
\end{align*}
Then we get:
\begin{thm}[Morgan-Szab\'o-Taubes]\label{thm MST}
In the situation above we have
\begin{equation*}
\sum_{k'\in \mathcal{K}(k)}SW_X(k')=\pm\sum SW_M(l_1)SW_N(l_2),
\end{equation*}
where the sum on the right extends over those $(l_1,l_2)\in \mathcal{K}_M(k)\times\mathcal{K}_N(k)$ with $\beta_X=\beta_M+\beta_N+2$.
\end{thm}
Note that if $M$ and $N$ are of simple type, then there is for a given characteristic class $k$ at most one non-zero summand on the right hand side of the formula; see also \cite{MuW}.

Finally, if $M$ and $N$ are symplectic and the surfaces $\Sigma_M$ and $\Sigma_N$ symplectically embedded, then the fibre sum $X$ admits a symplectic structure for all gluing diffeomorphisms $\phi$ as above \cite{Go, McW}. There is a formula for the canonical class of $X$ that in general depends on the choice of $C$. Let $X_0$ denote the fibre sum with $C=0$, corresponding to the gluing diffeomorphism that identifies the push-offs of the surfaces.
\begin{thm}\label{thm on the canonical class Gompf} Choose a normal form basis for $H_2(X;\mathbb{Z})$ as in Theorem \ref{formula H^2 no torsion}. Suppressing Poincar\'e duality, the canonical class of $X$ is given by
\begin{equation}\label{formula canonical class with sigma term}
K_X=\overline{K_M}+\overline{K_N}+\sum_{i=1}^cr_iR_i+b_XB_X+\sigma_X\Sigma_X,
\end{equation}
where
\begin{align*}
\overline{K_M}&=K_M-(2g-2)B_M-(K_MB_M-(2g-2)B_M^2)\Sigma_M \in P(M)\\
\overline{K_N}&=K_N-(2g-2)B_N-(K_NB_N-(2g-2)B_N^2)\Sigma_N \in P(N)\\
r_i&=K_XS_i=K_{X_0}S_i-a_i(K_NB_N+1-(2g-2)B_N^2)\\
b_X&=2g-2\\
\sigma_X&=K_MB_M+K_NB_N+2-(2g-2)(B_M^2+B_N^2).
\end{align*}
In evaluating $K_{X_0}S_i$, we choose the basis of rim tori in $X_0$ determined by the basis in $X$ and a corresponding dual basis of vanishing surfaces.
\end{thm}
The integers $a_i$ are defined as follows: Let $\alpha_1,\ldots,\alpha_c$ denote a basis for the kernel of the homomorphism $i_M\oplus i_N$ on $H_1(\Sigma;\mathbb{Z})$. Then
\begin{equation*}
a_i=\langle C,\alpha_i\rangle.
\end{equation*}

\section{Invariants of the iterated fibre sums $M(n)$}\label{sect invariants}

Let $M'$ denote an algebraic surface with a Lefschetz pencil and $H_1(M';\mathbb{Z})=0$ and $M\rightarrow\mathbb{CP}^1$ the algebraic Lefschetz fibration on the blow-up with fibre of genus $g$. Using two copies $M_1$ and $M_2$ of $M$ we see that the homomorphism
\begin{equation*}
i_{M_1}\oplus i_{M_2}
\end{equation*}
on $H_1(\Sigma;\mathbb{Z})$ maps to zero, hence
\begin{equation*}
H_1(M(2);\mathbb{Z})=0
\end{equation*}
and by induction
\begin{equation*}
H_1(M(n);\mathbb{Z})=0
\end{equation*}
for all $n\geq 2$. We can also consider a twisted fibre sum
\begin{equation*}
M(m,n,C)=M(m)\#_{\Sigma=\Sigma}M(n),
\end{equation*}
defined by a gluing diffeomorphism that is determined by a cohomology class $C$ in $H^1(\Sigma;\mathbb{Z})$. There is a diffeomorphism
\begin{equation*}
M(m,n,0)\cong M(m+n).
\end{equation*}
The same argument as above shows that
\begin{equation*}
H_1(M(m,n,C);\mathbb{Z})=0
\end{equation*}
for all $m,n\geq 1$. 

More specifically, if $M'$ is simply-connected, then the complement of the tubular neighbourhood of a general fibre in $M$ is also simply-connected, because the meridian to the surface bounds an embedded disk, coming from one of the exceptional spheres. Hence all fibre sums $M(n)$ and $M(m,n,C)$ are simply-connected.

The Euler characteristic and signature of the fibre sum $X=M(n)$ are given by
\begin{align*}
e(X)&=ne(M)+(n-1)(4g-4)\\
\sigma(X)&=n\sigma(M),
\end{align*}
and similarly for the twisted fibre sum. To describe the second homology of $M(2)$ we choose in $M$ for the surface $B_M$ one of the exceptional spheres, which form sections for the fibration. In $X=M(2)$ these spheres sew together to define a sphere $B_X$ of self-intersection $-2$. According to the second Lefschetz Hyperplane theorem the vanishing cycles generate the first homology of the general fibre. By our choice of gluing, the corresponding vanishing disks pair up in both summands of the fibre sum to define a set of vanishing classes given by embedded $(-2)$-spheres, called vanishing spheres. With Proposition \ref{gen set for S'X} we get:
\begin{prop}
There exists a basis $S_1,\ldots,S_{2g}$ for the group $S'(M(2))$ of vanishing surfaces which are linear combinations of embedded vanishing spheres and rim tori.
\end{prop}
Hence the intersection form of $X=M(2)$ looks like
\begin{equation*}
H_2(M(2);\mathbb{Z})=P(M_1)\oplus P(M_2)\oplus \bigoplus_{i=1}^{2g}\left(\begin{array}{cc} S_i^2 &1 \\ 1& 0 \\ \end{array}\right)\oplus\left(\begin{array}{cc} -2 &1 \\ 1& 0 \\ \end{array}\right).
\end{equation*}
By induction we get
\begin{equation*}
H_2(M(n);\mathbb{Z})=\bigoplus_{i=1}^n P(M_i)\oplus \bigoplus_{i=1}^{2g(n-1)}\left(\begin{array}{cc} S_i^2 &1 \\ 1& 0 \\ \end{array}\right)\oplus\left(\begin{array}{cc} -n &1 \\ 1& 0 \\ \end{array}\right).
\end{equation*}
In the induction step we perturb the fibration on $M(n-1)$ without changing the vanishing cycles such that the vanishing spheres become disjoint from the singular fibres. The terms in the middle part of the formula are the vanishing surface and dual rim tori pairs. The class $B_X$ in $M(n)$ is represented by a symplectic sphere of self-intersection $-n$, which is a section of the fibration. Note that the second Betti number has the value expected from the formula for the Euler characteristic. A similar formula holds for the fibre sums $M(m,n,C)$. However, in this case the vanishing surfaces are in general no longer linear combinations of vanishing spheres and rim tori.

We want to determine the canonical class of the symplectic fibre sum $M(n)$.
\begin{thm}\label{canonical class for M(n) Lefsch} The canonical class of $X=M(n)$ is given by
\begin{equation*}
K_X=\sum_{i=1}^n\overline{K_{M_i}}+(2g-2)B_X+((n-2)+(2g-2)n)\Sigma_X,
\end{equation*}
where 
\begin{equation*}
\overline{K_{M_i}}=(K_M+\Sigma_M)-(2g-2)(B_M+\Sigma_M)\in P(M_i)
\end{equation*}
for all $i=1,\dotsc,n$. 
\end{thm}
\begin{proof} The proof is by induction using the formula in Theorem \ref{thm on the canonical class Gompf}. We first show that all rim tori coefficients are zero. This is equivalent to showing that $K_XS_i=0$ for all vanishing classes $S_i$. Note that the vanishing classes are linear combinations of vanishing spheres and rim tori. Hence we only have to show that $K_XV=0$ for all vanishing spheres $V$ and $K_XT=0$ for all rim tori $T$. The sphere $V$ has a dual rim torus $R$ that intersects it in a single positive transverse point. Smoothing the intersection between $V$ and the rim torus $R$ we get an embedded torus of square zero. By the adjunction inequality and since $K_X$ is a Seiberg-Witten basic class according to a theorem of Taubes \cite{T}, we have for tori of square zero $K_XR=0$ and $K_X(V+R)=0$, hence $K_XV=0$. By the same argument $K_XT=0$. 

We first check the case $n=1$ of the formula we want to prove. We have:
\begin{align*}
K_X&=(K_M+\Sigma_M)-(2g-2)(B_M+\Sigma_M)\\
&+(2g-2)B_M+(-1+(2g-2))\Sigma_M\\
&=K_M.
\end{align*}
Suppose that $n\geq 2$ and the formula is correct for $n-1$. Write $N=M(n-1)$ and consider the fibre sum $X=M\#_{\Sigma_M=\Sigma_N}N$. The surface $B_M$ is an exceptional sphere in $M$ and $B_N$ is a symplectic sphere of self-intersection $-(n-1)$ from the previous step. Using the adjunction formula we have $K_MB_M=-1$, hence we get for $\overline{K_M}$, defined in Theorem \ref{thm on the canonical class Gompf},
\begin{align*}
\overline{K_M}&=K_M-(2g-2)B_M-(-1+(2g-2))\Sigma_M\\
&=(K_M+\Sigma_M)-(2g-2)(B_M+\Sigma_M)\\
&=\overline{K_{M_n}},
\end{align*}
and similarly $K_NB_N=n-3$, hence
\begin{align*}
\overline{K_N}&=K_N-(2g-2)B_N-((n-3)+(2g-2)(n-1))\Sigma_N\\
&=\sum_{i=1}^{n-1}\overline{K_{M_i}}+(2g-2)B_N+((n-3)+(2g-2)(n-1))\Sigma_N\\
&\quad -(2g-2)B_N-((n-3)+(2g-2)(n-1))\Sigma_N\\
&=\sum_{i=1}^{n-1}\overline{K_{M_i}}.
\end{align*}
We also have 
\begin{align*}
b_X&=2g-2\\
\sigma_X&=-1+(n-3)+2-(2g-2)(-1-(n-1))\\
&=(n-2)+(2g-2)n.
\end{align*}
Adding the terms proves the claim.
\end{proof} 
\begin{rem}\label{rem elliptic KM+SigmaM}
Note that for $g=1$ and $M$ equal to the elliptic surface $E(1)$ with general fibre $F$, we have $K_M+\Sigma_M=-F+F=0$. Hence we get the well-known formula $K_X=(n-2)F$ for the canonical class of $X=E(n)$.
\end{rem}
In the general case we have for the (maximal) divisibility of the canonical class:
\begin{cor}\label{divisibility of K_X} The divisibility of the canonical class $K_X$ of $X=M(n)$ is the greatest common divisor of $n-2$ and the divisibility of the class $K_M+\Sigma_M\in H^2(M;\mathbb{Z})$.
\end{cor}
\begin{proof} The greatest common divisor of $n-2$ and the divisibility of $K_M+\Sigma_M$ divides $K_X$. This follows because this number also divides $2g-2=(K_M+\Sigma_M)\Sigma_M$ by the adjunction formula. The number then divides all terms in the formula in Theorem \ref{canonical class for M(n) Lefsch}. 

Conversely, let $\delta$ denote the divisibility of $K_X$. It is clear that $\delta$ divides $2g-2$, since $K_X\Sigma_X=2g-2$ by the adjunction formula. We have
\begin{equation*}
K_XB_X=n-2,
\end{equation*}
since $B_X$ is a symplectic sphere of self-intersection $-n$. This implies that $\delta$ also divides $n-2$. The integer $\delta$ also has to divide every term $\overline{K_{M_i}}$. This shows that it divides the class $K_M+\Sigma_M$, proving the claim.
\end{proof} 
\begin{rem}
Since the complex curve $\Sigma_M$ in the blow-up $M=M'\#r{\overline{\mathbb{CP}}{}^{2}}\rightarrow\mathbb{CP}^1$ is the proper transform of a curve $\Sigma_{M'}$ in $M'$, the divisibility of $K_M+\Sigma_M$ is equal to the divisibility of $K_{M'}+\Sigma_{M'}$. In fact, both classes are equal because the canonical class and the class of the proper transform are given by
\begin{align*}
\Sigma_M&=\Sigma_{M'}-E_1-\dotsc-E_r\\
K_M&=K_{M'}+E_1+\dotsc+E_r,
\end{align*}
where $E_i$ denotes the exceptional spheres. 
\end{rem}
We can also determine the canonical class of the fibre sum $M(m,n,C)$.
\begin{thm}
The canonical class of $X=M(m,n,C)$ is given by
\begin{equation*}
K_X=\sum_{i=1}^{m+n}\overline{K_{M_i}}+\sum_{i=1}^{2g}r_iR_i+(2g-2)B_X+((m+n-2)+(2g-2)(m+n))\Sigma_X,
\end{equation*}
where 
\begin{equation*}
\overline{K_{M_i}}=(K_M+\Sigma_M)-(2g-2)(B_M+\Sigma_M)\in P(M_i)
\end{equation*}
for all $i=1,\dotsc,m+n$ and
\begin{equation*}
r_i=-a_i((2g-1)n-1)
\end{equation*}
for the coefficients $a_i=\langle C,\alpha_i\rangle$, determined by the class $C$ for a basis $\alpha_1,\ldots,\alpha_{2g}$ of $H_1(\Sigma;\mathbb{Z})$.
\end{thm}
\begin{proof}
The proof follows from the formula in Theorem \ref{thm on the canonical class Gompf}, because $X_0=M(m+n)$, hence $K_{X_0}S_i=0$, and $B_N$ is a symplectic sphere of self-intersection $-n$.
\end{proof}
The following is an immediate consequence which we will use later in deriving an obstruction to extending diffeomorphisms: 
\begin{cor}\label{div for M(m,n,C)}
The divisibility of the canonical class $K_X$ of $X=M(m,n,C)$ is the greatest common divisor of $m+n-2$, $a((2g-1)n-1)$ and the class $K_M+\Sigma_M$, where $a$ denotes the divisibility of the class $C$.
\end{cor}
Note that the divisibility of the class $C$ is the greatest common divisor of the integers $a_i$. We can use this corollary for example to determine when the manifold $M(m,n,C)$ is spin.
\begin{rem}
One can show that the divisibility is also equal to the greatest common divisor of $m+n-2$, $a((2g-1)m-1)$ and $K_M+\Sigma_M$, as required by symmetry.
\end{rem}
Finally, we determine some of the Seiberg-Witten invariants of the manifolds $M(n)$ in the special case that $M'$ is a minimal surface of general type. Surfaces of general type are algebraic. This follows from Theorem 6.2 in Chapter IV and Theorem 2.2 in Chapter VII in  \cite{BHPV}. We have the following lemma. 
\begin{lem}\label{lem genus general type at least 2}
Let $M'$ be a smooth minimal surface of general type embedded in some complex projective space and $\Sigma_{M'}$ a transverse hyperplane section. Then the genus of $\Sigma_{M'}$ is at least two.
\end{lem}
\begin{proof}
We have $\Sigma_{M'}^2=r\geq 1$, where $r$ denotes the degree of $M'$. Since $M'$ is minimal and of general type, we have $K_{M'}C\geq 0$ for all smooth curves $C$ with equality if and only if $C$ is a rational $(-2)$-curve \cite[Chapter VII, Corollary 2.3]{BHPV}. Hence $K_{M'}\Sigma_{M'}\geq 1$ and the adjunction formula implies that the genus of $\Sigma_{M'}$ is at least two.
\end{proof}
For simplicity we also assume in the following that $b_2^+(M')>1$. We can therefore use the formula from Theorem \ref{thm MST}. First note that the Poincar\'e dual of a basic class of $M(n)$ has no rim tori component: This follows as before in the proof of Theorem \ref{canonical class for M(n) Lefsch} from the adjunction inequality of Seiberg-Witten theory because the vanishing surfaces are linear combinations of  vanishing spheres and rim tori; see \cite{FS2} for a similar argument. Hence the left hand side in Theorem \ref{thm MST} has only one term. It follows from \cite{MuW} that the intersection of a basic class of $M(n)$ with $\Sigma_X$ is either equal to $\pm(2g-2)$ or zero.
\begin{thm}
Suppose that $M'$ is a minimal algebraic surface of general type with $H_1(M';\mathbb{Z})=0$ and $b_2^+>1$. Then the only Seiberg-Witten basic class up to sign of $X=M(n)$, which has non-zero intersection with the fibre $\Sigma_X$, is the canonical class $K_X$.
\end{thm}
\begin{proof}
We consider the argument for $M(2)$. The general case follows similarly. 
By the blow-up formula for the Seiberg-Witten invariants \cite{FS} and the calculation of the basic classes for a minimal surface of general type \cite{W}, the basic classes of $M=M'\#r{\overline{\mathbb{CP}}{}^{2}}$ are given by 
\begin{equation*}
L=\pm (K_{M'}\pm E_1\pm\ldots \pm E_r).
\end{equation*}
Note that on the right hand side of the Morgan-Szab\'o-Taubes formula only the basic classes $L$ on $M$ with $L\Sigma_M=2g-2$ are relevant. Suppose that $i$ exceptional sphere summands in the bracket of $L$ have negative sign and $r-i$ positive. We have 
\begin{equation*}
2g-2=K_{M'}\Sigma_{M'}+r.
\end{equation*}
Since $M'$ is minimal and of general type and $\Sigma_{M'}$ is a smooth complex curve of non-zero genus, we have $K_{M'}\Sigma_{M'}>0$ by the proof of Lemma \ref{lem genus general type at least 2}, which implies $i\leq r<2g-2$. We get:
\begin{equation*}
|L\Sigma_M|=|K_{M'}\Sigma_{M'}+r-2i|=|2g-2-2i|.
\end{equation*}
This can be $2g-2$ if and only if $i=0$. Hence the relevant basic class of $M$ for the right hand side of the Morgan-Szab\'o-Taubes formula is $L=K_M$. We then see that the only basic class of $M(2)$ which has intersection $2g-2$ with the fibre is the canonical class.
\end{proof}

\section{Extension of diffeomorphisms}

Suppose that $M'$ is a minimal surface of general type with $H_1(M';\mathbb{Z})=0$ and $b_2^+>1$. Let $M(n)$ be a fibre sum as above and $\Sigma$ a general fibre in $M(n)$ with tubular neighbourhood $\nu\Sigma$. The neighbourhood has a natural framing $\Sigma\times D^2$, given by the fibration. We are interested in orientation preserving self-diffeomorphisms $\psi$ of $\partial\nu\Sigma$ which preserve the circle fibration and cover the identity. As before, the diffeomorphism $\psi$ has up to isotopy the form
\begin{align*}
\psi\colon\Sigma\times S^1&\longrightarrow\Sigma\times S^1\\
(x,\alpha)&\mapsto (x,C(x)\cdot\alpha),
\end{align*}
where $C\colon\Sigma\rightarrow S^1$ is a smooth map. The diffeomorphism is determined up to isotopy by the cohomology class $C^*d\alpha\in H^1(\Sigma;\mathbb{Z})$, denoted by $C$. 

We want to derive an obstruction so that $\psi$ does not extend to an orientation preserving self-diffeomorphism of the complement $M(n)\setminus\text{int}\,\nu\Sigma$.   
\begin{thm}\label{obstruction diffeom}
Let $d$ denote the divisibility of the class $K_{M'}+\Sigma_{M'}$ and $a$ the divisibility of the class $C$. If $d$ does not divide $a(n-1)$, then $\psi$ does not extend to an orientation preserving self-diffeomorphism of the complement $M(n)\setminus\text{int}\,\nu\Sigma$.
\end{thm}
\begin{proof}
Suppose that $\psi$ extends to a self-diffeomorphism $\Psi$ of the complement. We can use $\psi$ to form a twisted fibre sum $M(m,n,C)$ for any integer $m\geq 1$. Since $\psi$ extends on the complement there is an orientation preserving diffeomorphism $M(m,n,C)\cong M(m+n)$. Let $X'=M(m,n,C)$ and $X=M(m+n)$. Note that the diffeomorphism is the identity on $M(m)_0$. Hence it maps the surface $\Sigma_{X'}$ in $X'$, given by the standard fibre in the boundary of the tubular neighbourhood in $M(m)$, to the standard fibre $\Sigma_X$ in $M(m+n)$. On the $M(n)_0$ side in $X'$, the surface $\Sigma_{X'}$ is a twisted copy of the standard fibre and also gets mapped to the fibre $\Sigma_X$ in $X$. The canonical class of $X'$ has intersection $2g-2$ with the symplectic surface $\Sigma_{X'}$. It has to map under the diffeomorphism to a Seiberg-Witten basic class of $X$ having the same intersection $2g-2$ with the fibre $\Sigma_X$, since the surfaces get identified and the canonical class is basic \cite{T}. There is a unique such basic class, the canonical class of $X$. It follows that both canonical classes have to match under the diffeomorphism, in particular they must have the same divisibility. This will imply that $d$ divides $a(n-1)$.

The divisibility of the canonical class of $X$ is the greatest common divisor of
\begin{equation*}
m+n-2\text{ and } d.
\end{equation*}
The divisibility of the canonical class of $X'$ is the greatest common divisor of
\begin{equation*}
m+n-2, d\text{ and }a((2g-1)n-1).
\end{equation*}
If both are the same, the greatest common divisor of $m+n-2$ and $d$ must divide $a((2g-1)n-1)$. We can choose for $m$ any positive integer. In particular, we can arrange that the greatest common divisor of $m+n-2$ and $d$ is equal to $d$. Hence $d$ has to divide $a((2g-1)n-1)$. The integer $d$ is the divisibility of $K_{M'}+\Sigma_{M'}$. The adjunction formula shows that $d$ divides $2g-2$. It follows that $d$ divides $a((2g-1)n-1)$ if and only if it divides $a(n-1)$. This completes the proof. 
\end{proof}
\begin{rem}
The corresponding statement is true for elliptic surfaces and shows that if a diffeomorphism extends over $E(n)\setminus\text{int}\,\nu\Sigma$ then $d$ divides $a(n-1)$. Note that in the case of $E(n)$ the integer $d$ is zero, see Remark \ref{rem elliptic KM+SigmaM}. This implies that we must have either $n=1$, and we are in the case of $E(1)$, or $a=0$ and hence $C=0$, which is the case of a trivial diffeomorphism that preserves the torus fibration. Hence we get the same obstruction that is known for elliptic surfaces, see \cite[Theorem 8.3.11]{GS}.
\end{rem}

\section{Construction of some examples}

Let $M'$ be a minimal complex surface of general type. Then $M'$ is algebraic and an embedding into some projective space $\mathbb{CP}^N$ is determined by a very ample line bundle $E$ on $M'$. Under such an embedding, a transverse hyperplane section of $M'$ will be a surface $\Sigma_{M'}$ on $M'$, representing the Poincar\'e dual of $c_1(E)$. We want to prove the following:
\begin{prop}\label{existence of emb overline{M} div by d}
For each integer $d\geq 1$ there exists an embedding of $M'$ into some projective space $\mathbb{CP}^N$ such that the class $K_{M'}+\Sigma_{M'}$ is divisible by $d$.
\end{prop}
Using this proposition we can construct many Lefschetz fibrations $M(n)$ such that certain diffeomorphisms on the boundary of the tubular neighbourhood of a general fibre do not extend over the complement according to Theorem \ref{obstruction diffeom}.

We need some preparations: If $S$ is an ample line bundle on $M'$, then $kS$ is very ample for all $k\geq k_0$ for some integer $k_0$ \cite[Theorem 1.2.6]{Laz}. Let $L$ be an ample line bundle on $M'$. We want to determine when the line bundle of the form $S=K_{M'}+sL$ for integers $s\geq 1$ is ample. 
\begin{lem} Let $L$ be an ample line bundle on $M'$. Then there exists an integer $s_0 \geq 1$ such that $S=K_{M'}+sL$ is ample for all $s\geq s_0$.
\end{lem}
\begin{proof}
According to the Nakai-Moishezon criterion, the line bundle $S$ is ample if and only if $S^2>0$ and $SC>0$ for every irreducible curve $C$ on $M'$; see \cite[Chapter IV, Corollary 6.4]{BHPV} and \cite[Chapter V, Theorem 1.10]{H}. The canonical class of an algebraic surface of general type satisfies $K_{M'}C\geq 0$ for all irreducible curves $C$, with equality if and only if $C$ is a rational $(-2)$-curve. Since we assumed that $L$ is ample, we have $SC>0$ for all integers $s\geq 1$ and every curve $C$. Moreover, since $L$ is ample we also have $L^2>0$. Hence the second condition $S^2=K_{M'}^2+2sK_{M'}L+s^2L^2>0$ will certainly hold if we choose $s$ large enough.
\end{proof}
We can now prove Proposition \ref{existence of emb overline{M} div by d}.
\begin{proof}
Let $S=K_{M'}+sL$ be ample and $kS$ very ample. Then the hyperplane section $\Sigma_{M'}$ under the embedding determined by $kS$ into a projective space $\mathbb{CP}^N$ represents the class $kK_{M'}+ksL$. Let $d\geq 1$ be some integer. We want to choose $s$ and $k$ such that the class $K_{M'}+\Sigma_{M'}$ is divisible by $d$. First choose $s$ such that $S=K_{M'}+sL$ is ample and $d$ divides $s$. Then choose $k$ large enough such that $kS$ is very ample and $d$ divides $k+1$. Then $d$ also divides the class $K_{M'}+\Sigma_{M'}$.
\end{proof}

\bibliographystyle{amsplain}

\bigskip
\bigskip

\end{document}